\documentclass{amsart}
\usepackage{latexsym,amsmath,amssymb}
\newtheorem{thm}{Theorem}[section]
\newtheorem{cor}[thm]{Corollary}

\newtheorem{prop}[thm]{Proposition}
\theoremstyle{definition}\newtheorem{defn}[thm]{Definition}
\theoremstyle{remark}

\numberwithin{equation}{section}

\begin{document}

\title[]
{Some classes of composition operators on Orlicz spaces}

\author{\sc\bf   Z. Huang and Y. Estaremi }
\address{ \sc Z. Huang}
\email{jameszhuang923@gmail.com}
\address{Huxley Building Department of Mathematics, South Kensington Campus,Imperial College London, London, UK.}
\address{\sc Y. Estaremi}
\email{y.estaremi@gu.ac.ir}
\address{Department of Mathematics, Faculty of Sciences, Golestan University, Gorgan, Iran}


\thanks{}

\thanks{}

\subjclass[2020]{47B33, 37B65}

\keywords{Orlicz spaces, Composition operators, Expansivity, structural stability.}

\date{}

\dedicatory{}

\commby{}

\begin{abstract}
The notions of expansivity and positive expansivity for composition operators on Orlicz spaces are investigated. In particular, necessary and sufficient conditions are given for a composition operator to be expansive, positively expansive, and uniformly expansive. Additionally, equivalent conditions for these concepts are provided in the case that the system is dissipative.
\end{abstract}

\maketitle

\section{ \sc\bf Introduction and Preliminaries}
An Orlicz space $L^{\Phi}$ is known to be a natural generalization of $L^p$ spaces and has been considered
 in various areas such as probability theory, partial differential equations, and mathematical finance.
 The dynamics of bounded linear operators on infinite-dimensional Banach spaces have been studied
 for decades and applied to obtain meaningful results in other areas of mathematics such as number
 theory, ergodic theory, and geometry of Banach spaces. The study of composition operators in linear
 dynamics has been considered by many researchers because such a class is concrete and large enough to
 generate various examples and counterexamples. Interested readers can refer to \cite{bm,br,bmpp,gm}.

In \cite{eh}, Eisenberg and Hedlund investigated the relationships between expansivity and the spectrum of operators on Banach spaces. Additionally, in \cite{nbpc}, the authors established relationships between notions of expansivity and hypercyclicity, supercyclicity, Li–Yorke
 chaos, and shadowing. In the case that the Banach space is $c_0$ or $l_p$ $(1\leq p<\infty)$, they characterized all weighted shift operators which satisfy various notions of
 expansivity. In \cite{mm}, the authors studied the two notions of expansivity and strong structural stability for composition operators on $L^p(\mu)$-spaces, $1\leq p<\infty$. They provided some necessary and sufficient conditions for composition operators to be expansive both in the general and dissipative case. They also showed that, in the dissipative setting, the shadowing
 property implies strong structural stability and provided that these two notions
 are equivalent under some extra hypothesis of positive expansivity. Many basic properties of composition operators on $L^p$ and Orlics spaces have been studied by mathematicians in recent years. One can see for example, \cite{aa,af,aj, j,jh}.

In this paper, we focus on composition operators on Orlicz spaces, whose boundedness was characterized
 in \cite{chkm} by Cui, Hudzik, Kumar, and Maligranda with the growth condition on Young functions called $\Delta_2$. First, we extend the definitions of dissipative systems and bounded distortion property to the setting of Orlicz spaces in such a way that they coincide with the definitions of these concepts in $L^p(\mu)$-spaces, in the special case. Then, we provide some necessary and sufficient conditions for expansivity, positive expansivity, and uniform expansivity of such operators. Moreover, we provide some equivalent conditions for these concepts in the case that the system is dissipative.

In the sequel, for the convenience of the reader, we recall some essential facts on Orlicz spaces for later use. For more details on Orlicz spaces, see \cite{kr,raor}.

A function $\Phi:\mathbb{R}\rightarrow [0,\infty]$ is called a \textit{Young's function}  if $\Phi$ is   convex, $\Phi(-x)=\Phi(x)$, $\Phi(0)=0$ and $\lim_{x\rightarrow \infty} \Phi(x)=+\infty$. Each \textit{Young's function} $\Phi$, has an associate function $\Psi:\mathbb{R} \rightarrow [0, \infty]$ defined by

$$\Psi(y)=\sup\{x\mid y\mid-\Phi(x): x\geq0\}, \ \ \ \ \ y\in \mathbb{R},$$

which is called the \textit{complementary function} to $\Phi$. This function is increasing, convex, $\Psi(0)=0$, $\Psi(-y)=\Psi(y)$ and $\lim_{y\rightarrow \infty} \Psi(y)=+\infty$.

A \textit{Young's function} $\Phi$ satisfies the
$\Delta_{2}$-condition or it is $\Delta_2$-regular, if $\Phi(2x)\leq
K\Phi(x) \; ( x\geq x_{0})$  for some constants
$K>0$ and $x_0>0$. Also, $\Phi$ is said to satisfy the
$\Delta'$ condition, if $\exists c>0$
$(b>0)$ such that
$$\Phi(xy)\leq c\Phi(x)\Phi(y), \ \ \ x,y\geq x_{0}\geq 0.$$
If $x_0 = 0$, then it is said to hold
 globally. If $\Phi \in \Delta'$, then $\Phi \in \Delta_2$.

For a given complete $\sigma$-finite measure space $(X, \mathcal{F}, \mu)$, let $L^0(\mathcal{F})$ be the linear space of equivalence classes of $\mathcal{F}$-measurable real-valued functions on $X$, where functions that are equal $\mu$-almost everywhere on $X$ are identified. The support $S(f)$ of a
 measurable function $f$ is defined as $S(f) := \{x \in X : f(x) \neq 0\}$.

For a \textit{Young's function} $\Phi$, let $\rho_{\Phi}:L^{\Phi}(\mu) \rightarrow \mathbb{R}^{+}$ be defined as $\rho_{\Phi}(f) = \int_{X} \Phi(f) d\mu$ for all $f \in L^{\Phi}(\mu)$. Then the space
$$
L^{\Phi}(\mu) = \left\{f \in L^0(\mathcal{F}) : \exists k > 0, \rho_{\Phi}(kf) < \infty\right\}
$$
is called an Orlicz space. The functional

$$N_{\Phi}(f) = \inf \{k > 0 : \rho_{\Phi}(\frac{f}{k}) \leq 1\}$$
is a norm on $L^{\Phi}(\mu)$ and is called the \textit{gauge norm} (or Luxemburg norm). Also, $(L^{\Phi}(\mu), N_{\Phi}(.))$ is a normed linear space. If almost everywhere equal functions are identified, then $(L^{\Phi}(\mu), N_{\Phi}(.))$ is a Banach space, and the basic measure space $(X, \mathcal{F}, \mu)$ is unrestricted. Hence, every element of $L^{\Phi}(\mu)$ is a class of measurable functions that are almost everywhere equal. Additionally, there is another norm on $L^{\Phi}(\mu)$, defined as follows:

%
%
%
%
\begin{equation}\label{e1}
\|f\|_{\Phi}=\sup\{\int_{X}\mid fg\mid d\mu: g\in B_{\Psi}\}=\sup\{\mid\int_{X}fg d\mu\mid: g\in B_{\Psi}\},
\end{equation}
in which $B_{\Psi}=\{g\in L^{\Psi}(\mu): \int_{X}\Psi(\mid g\mid )d\mu\leq 1\}$.

The norm $\|.\|_{\Phi}$ is called the Orlicz norm. For any $f\in L^{\Phi}(\mu)$, where $\Phi$ is a Young function, we have:

\begin{equation}\label{e2}
N_{\Phi}(f)\leq \|f\|_{\Phi}\leq2N_{\Phi}(f).
\end{equation}

For every $F\in\mathcal{F}$ with $0<\mu(F)<\infty$, we have $N_{\Phi}(\chi_F)=\frac{1}{\Phi^{-1}(\frac{1}{\mu(F)})}$. Furthermore, if $\Phi\in \Delta_2$, then $(L^{\Phi}(\mu))^*=L^{\Psi}(\mu)$, where $\Psi$ is the complementary Young's function to $\Phi$.

Now, we recall the following theorem about convergence of sequences in Orlicz spaces:

\begin{thm}\cite{raor}\label{t1.1}
Let $\{f_n\}_{n\geq1}$ be a sequence in $L^{\Phi}(\mu)$ and $f\in L^{\Phi}(\mu)$. Then the following assertions hold:

(a) If $\|f_n-f\|_{\Phi}\rightarrow 0$ (or equivalently $N_{\Phi}(f_n-f)\rightarrow 0$), then $\rho_{\Phi}(f_n)\rightarrow \rho_{\Phi}(f)$. The converse holds if $\Phi$ is $\triangle_2$-regular.

(b) If $\Phi$ is a $\triangle_2$-regular Young function, or if $\Phi$ is continuous and concave with $\Phi(0)=0$, $\Phi\nearrow$ as well, and $\rho_{\Phi}(f_n)\rightarrow \rho_{\Phi}(f)$ as $n\rightarrow \infty$ and $f_n\rightarrow f$ almost everywhere (a.e.), or in $\mu$-measure, then $f_n\rightarrow f$ in norm.
\end{thm}

Throughout this paper, $(X,\mathcal{F},\mu)$ will denote a measure space, where $X$ is a set, $\mathcal{F}$ is a sigma algebra on $X$, and $\mu$ is a positive measure on $\mathcal{F}$. Additionally, we assume that $\varphi:X\rightarrow X$ is a non-singular measurable transformation, meaning that $\varphi^{-1}(F)\in \mathcal{F}$ for every $F\in \mathcal{F}$ and $\mu(\varphi^{-1}(F))=0$ if $\mu(F)=0$. Moreover, if there exists a positive constant $c$ such that $\mu(\varphi^{-1}(F))\leq c\mu(F)$ for every $F\in \mathcal{F}$, then the linear operator

$$C_{\varphi}:L^{\Phi}(\mu)\rightarrow L^{\Phi}(\mu), \ \  \ \ \ \ C_{\varphi}(f)=f\circ\varphi,$$

is well-defined and continuous on the Orlicz space $L^{\Phi}(\mu)$, and is called a composition operator. For more details on composition operators on Orlicz spaces, one can refer to \cite{chkm}.

\section{\sc\bf Expansivity for composition operators}
In this section we provide some equivalent conditions for expansivity, positive expansivity, uniform expansivity and uniformly positive expansivity of a composition operator on the Orlicz space $L^{\Phi}(\mu)$. Additionally we find necessary and sufficient conditions for structural stability, strong structural stability and the shadowing property of composition operators. First we recall the following result from \cite{mm}.
\begin{prop}\label{p1.1}
Let $T$ be an operator on a Banach space $X$ and let  $S_X=\{x\in X: ||x||=1\}$ be the  unite sphere of $X$. Then:\\
(a) $T$ is positively expansive if and only if $\sup_{n\in \mathbb{N}}\|T^n x\|=\infty$, for every nonzero $x\in X$.\\
(b) $T$ is uniformly positively expansive if and only if $\lim_{n\rightarrow \infty}\|T^n x\|=\infty$, uniformly on $S_X$.\\

If, in addition, $T$ is invertible, then:\\
(c) $T$ is expansive if and only if $\sup_{n\in \mathbb{Z}}\|T^n x\|=\infty$, for every nonzero $x\in X$.\\
(d) $T$ is uniformly expansive if and only if $S_X=A\cup B$, where $\lim_{n\rightarrow \infty}\|T^n x\|=\infty$, uniformly on $A$ and
$\lim_{n\rightarrow \infty}\|T^{-n}x\|=\infty$, uniformly on $B$.
\end{prop}

Let $C_{\varphi}:L^{\Phi}(\mu)\rightarrow L^{\Phi}(\mu)$ be a composition operator and let $\mathcal{F}^+=\{F\in \mathcal{F}: 0<\mu(F)<\infty\}$. Here we rewrite the definition of composition dynamical system for composition operators on Orlicz spaces $L^{\Phi}(\mu)$.
\begin{defn}
A composition dynamical system is a quintuple $(X,\mathcal{F}, \mu, \varphi, C_{\varphi})$ where
\begin{enumerate}
\item $(X,\mathcal{F}, \mu)$ is a $\sigma$-finite measure space,
\item $\varphi:X\rightarrow X$ is an injective bimeasurable transformation, i.e., $\varphi(F)\in \mathcal{F}$ and $\varphi^{-1}(F)\in \mathcal{F}$ for
every $F\in\mathcal{F}$,
 \item There is $c>0$ such that
 \begin{equation}\label{ec}
 \mu(\varphi^{-1}(A)\leq c\mu(A) \ \ \ \ \ \ \text{for every} \ \ A\in\mathcal{F},
  \end{equation}

 \item $C_{\varphi}:L^{\Phi}(\mu)\rightarrow L^{\Phi}(\mu)$ is the composition operator induced by $\varphi$, defined as
  $$C_{\varphi}f=f\circ \varphi.$$
\end{enumerate}
It is known that the condition \ref{ec} is a necessary and sufficient condition for boundeness of $C_{\varphi}$ on the Orlicz space $L^{\Phi}(\mu)$. Additionally, if $\varphi$ is bijective and $\varphi^{-1}$ satisfies the condition \ref{ec} then $C_{\varphi^{-1}}$ is a bounded linear operators and $C^{-1}_{\varphi}=C_{\varphi^{-1}}$. For more details on composition operators we refer interested readers to \cite{sm}.
\end{defn}

Furthermore by the fact that for every $F\in\mathcal{F}$ with $0<\mu(F)<\infty$, we have $\|\chi_{F}\|_p=(\int_X|\chi_F|^pd\mu)^{\frac{1}{p}}=\mu(F)^{\frac{1}{p}}$ and $N_{\Phi}(\chi_F)=\frac{1}{\Phi^{-1}(\frac{1}{\mu(F)})}$, the definition of a wandering set and dissipative system of bounded distortion in $L^p(\mu)$-spaces were defined in \cite{ddm}. Here we rewrite these definitions in the context of Orlicz spaces.

\begin{defn}
Let $(X, \mathcal{F}, \mu)$ be a measure space and let $\varphi:X \rightarrow X$ be an invertible non-singular transformation. A measurable set $W\subseteq X$ is called a wandering set for $\varphi$ if the sets $\{\varphi^{-n}(W)\}_{n\in\mathbb{Z}}$ are pair-wise disjoint.
\end{defn}

\begin{defn}
Let $(X, \mathcal{F}, \mu)$ be a measure space and let $\varphi:X\rightarrow X$ be an invertible non-singular transformation. The quadruple $(X, \mathcal{F}, \mu, \varphi)$ is called
\begin{itemize}
  \item A dissipative system generated by $W$, if $X=\dot{\cup}_{k\in \mathbb{Z}}\varphi^k(W)$ for some $W\in \mathcal{F}$ with $0<\mu(W)<\infty$ (the symbol $\dot{\cup}$ denotes pairwise disjoint union);
  \item A dissipative system, of bounded distortion, generated by $W$, if there exists $K>0$, such that
  \begin{equation}\label{edis}
  \frac{1}{K}\frac{N_{\Phi}(C^k_{\varphi}(\chi_W))}{N_{\Phi}(\chi_W)}\leq \frac{N_{\Phi}(C^k_{\varphi}(\chi_F))}{N_{\Phi}(\chi_F)}\leq K\frac{N_{\Phi}(C^k_{\varphi}(\chi_W))}{N_{\Phi}(\chi_W)},
  \end{equation}
  for all $k\in \mathbb{Z}$ and $F\in \mathcal{F}_W=\{F\cap W: F\in \mathcal{F}\}$. If we replace $N_{\Phi}$ by the norm of $L_p$, then we will have the bounded distortion property in $L_p$-spaces.(Definition 2.6.4, \cite{ddm})
 \end{itemize}
\end{defn}
\begin{defn}
A composition dynamical system $(X, \mathcal{F}, \mu, \varphi, C_{\varphi})$ is called
\begin{itemize}
\item dissipative composition dynamical system, generated by $W$, if $(X, \mathcal{F}, \mu, \varphi)$ is a dissipative system generated by $W$;
\item dissipative composition dynamical system, of bounded distortion, generated by $W$, if $(X, \mathcal{F}, \mu, \varphi)$ is a dissipative system of bounded distortion, generated by $W$.
\end{itemize}

\end{defn}
Based on the above definitions we have the following Proposition similar to the same result for $L^P(\mu)$-spaces in \cite{eb}.
\begin{prop}\label{p2.6}
Let $(X, \mathcal{F}, \mu, \varphi)$ be a dissipative system of bounded distortion, generated by $W$. Then there exists $H>0$ such that
\begin{equation}\label{gedis}
\frac{1}{H}\frac{N_{\Phi}(C^{t+s}_{\varphi}(\chi_W))}{N_{\Phi}(C^s_{\varphi}(\chi_W))}\leq \frac{N_{\Phi}(C^{t+s}_{\varphi}(\chi_F))}{N_{\Phi}(C^s_{\varphi}(\chi_F))}\leq H\frac{N_{\Phi}(C^{t+s}_{\varphi}(\chi_W))}{N_{\Phi}(C^s_{\varphi(\chi_W)})},
\end{equation}
for all $F\in \mathcal{F}_W$ with $\mu(F)>0$ and all $s,t\in \mathbb{Z}$.
\end{prop}
Now in the next Theorem we obtain equivalent conditions for composition operators to be expansive, positively expansive and uniformly positively expansive.

\begin{thm}\label{t2.2}
Let $(X, \mathcal{F}, \mu, \varphi, C_{\varphi})$ be a composition dynamical system. The followings hold.\\
\begin{enumerate}
  \item The composition operator $C_{\varphi}$ is positively expansive if and only if for every $A\in \mathcal{F}$, with $0< \mu(A)<\infty$, we have
  $$\inf_{n\in \mathbb{N}}\Phi^{-1}(\frac{1}{\mu(\varphi^{-n}(A))})=0.$$\\
  \item The composition operator $C_{\varphi}$ is expansive if and only if for every $A\in \mathcal{F}$, with $0<\mu(A)<\infty$, we have
  $$\inf_{n\in \mathcal{Z}}\Phi^{-1}(\frac{1}{\mu(\varphi^{-n}(A))})=0.$$\\
  \item If $\Phi$ is $\Delta_2$-regular, then the composition operator $C_{\varphi}$ is uniformly positively expansive if and only if
  $$\lim_{n\rightarrow \infty}\frac{\Phi^{-1}(\frac{1}{\mu(A)})}{\Phi^{-1}(\frac{1}{\mu(\varphi^{-n}(A))})}=\infty,$$
  for all $A\in \mathcal{F}$, with $0<\mu(A)<\infty$.\\
  \item If $\Phi$ is $\Delta_2$-regular, then the composition operator is uniformly expansive if and only if $\mathcal{F}^+$ can be splitted as  $\mathcal{F}^+=\mathcal{F}^+_A\cup \mathcal{F}^+_B$ such that
      $$\lim_{n\rightarrow \infty}\frac{\Phi^{-1}(\frac{1}{\mu(F)})}{\Phi^{-1}(\frac{1}{\mu(\varphi^{n}(F))})}=\infty, \ \ \ \text{uniformly on } \ \mathcal{F}^+_A,$$
       $$\lim_{n\rightarrow \infty}\frac{\Phi^{-1}(\frac{1}{\mu(F)})}{\Phi^{-1}(\frac{1}{\mu(\varphi^{-n}(F))})}=\infty, \ \ \ \text{uniformly on } \ \mathcal{F}^+_B.$$

\end{enumerate}
\end{thm}
\begin{proof}
First we prove (2). Let $C_{\varphi}$ be expansive. Then by part (c) of the Proposition \ref{p1.1} we have
$$\sup_{n\in \mathbb{Z}}N_{\Phi}(C^n_{\varphi}f)=\sup_{n\in \mathbb{Z}}N_{\Phi}(C_{\varphi^n}f)=\infty,  \ \ \ \ \ \text{for each} \ \ \ f\in L^{\Phi}(\mu)\setminus \{0\}.$$
Now if $F\in \mathcal{F}$ with $0<\mu(F)<\infty$, then by setting $f=\chi_F$ and for every $n\in \mathcal{Z}$ we have
$$N_{\Phi}(C^n_{\varphi}f)=N_{\Phi}(\chi_{\varphi^{-n}(F)})=\frac{1}{\Phi^{-1}(\frac{1}{\mu(\varphi^{-n}(F))})}.$$
So we have
$$\sup_{n\in \mathbb{Z}}N_{\Phi}(C^n_{\varphi}f)=\sup_{n\in\mathbb{Z}}\frac{1}{\Phi^{-1}(\frac{1}{\mu(\varphi^{-n}(F))})}=\infty$$
and therefore
$$\inf_{n\in\mathbb{Z}}\Phi^{-1}(\frac{1}{\mu(\varphi^{-n}(F))})=0.$$
Now we prove the converse. Suppose that for each $F\in \mathcal{F}$, with $0<\mu(F)<\infty$, we have $\inf_{n\in \mathbb{Z}}\Phi^{-1}(\frac{1}{\mu(\varphi^{-n}(F))})=0$. If $f\in L^{\Phi}(\mu)\setminus \{0\}$, then there exists a positive constant $\delta$ such that the set $F_{\delta}=\{x\in X: |f(x)|>\delta\}$ has a positive measure, i.e, $\mu(F_{\delta})>0$. Since $(X, \mathcal{F}, \mu)$ is $\sigma$-finite, then without loss of generality we can assume that $\mu(F_{\delta})<\infty$. So for every $n\in \mathbb{Z}$:
\begin{align*}
\int_X\Phi(\frac{\delta.\chi_{F_{\delta}}\circ \varphi^n}{N_{\Phi}(C^n_{\varphi})})d\mu&\leq \int_{\varphi^{-n}(F_{\delta})}\Phi(\frac{f\circ \varphi^n}{N_{\Phi}(C^n_{\varphi})})d\mu\\
&\leq \int_{X}\Phi(\frac{f\circ \varphi^n}{N_{\Phi}(C^n_{\varphi})})d\mu\leq1.
\end{align*}
Hence $N_{\Phi}(\delta.\chi_{F_{\delta}})\leq N_{\Phi}(C^n_{\varphi}f)$ and so
$$\delta. \frac{1}{\Phi^{-1}(\frac{1}{\mu(\varphi^{-n}(F_{\delta}))})}\leq N_{\Phi}(C^n_{\varphi}f).$$
This inequality implies that if
 $$\inf_{n\in \mathbb{Z}}\Phi^{-1}(\frac{1}{\mu(\varphi^{-n}(F_{\delta}))})=0,$$
  then $\sup_{n\in\mathbb{Z}}N_{\Phi}(C^n_{\varphi}f)=\infty$. Since $f$ is chosen arbitrarily, then by part (c) of the Proposition \ref{p1.1} we get that $C_{\varphi}$ is expansive on $L^{\Phi}(\mu)$.\\
  For the proof of part (1) it suffices just to replace $\mathbb{Z}$ by $\mathbb{N}$.\\
  (3). Let $C_{\varphi}$ be uniformly positively expansive. Then by Proposition \ref{p1.1} we have
  $$\lim_{n\rightarrow \infty}N_{\Phi}(C^n_{\varphi}f)=\infty, \ \ \ \ \text{uniformly on} \ \ \ \ \  S_{L^{\Phi}(\mu)},$$
  in which $S_{L^{\Phi}(\mu)}=\{f\in L^{\Phi}(\mu): N_{\Phi}(f)=1\}$.
Let $f=\Phi^{-1}(\frac{1}{\mu(F)})\chi_{F}$, for each $F\in \mathcal{F}^+$. Then we have $N_{\Phi}(f)=1$ and also\\
$$N_{\Phi}(C^n_{\varphi}f)=\Phi^{-1}(\frac{1}{\mu(F)})N_{\Phi}(\chi_{\varphi^{-n}(F)})=\frac{\Phi^{-1}(\frac{1}{\mu(F)})}{\Phi^{-1}(\frac{1}{\mu(\varphi^{-n}(F))})}.$$
So by our assumptions we get that
$$\lim_{n\rightarrow \infty}\frac{\Phi^{-1}(\frac{1}{\mu(F)})}{\Phi^{-1}(\frac{1}{\mu(\varphi^{-n}(F))})}=\infty, \ \ \ \ \ \text{uniformly on the sets} \ \ \ F\in \mathcal{F}^+.$$
Conversely, since $\Phi\in \Delta_2$, then the simple functions are dense in $L^{\Phi}(\mu)$\cite{raor}. So it suffices to prove that for any simple function $f\in S_{L^{\Phi}(\mu)}$, the equality $$\lim_{n\rightarrow \infty}N_{\Phi}(C^n_{\varphi}F)=\infty$$ holds. By our assumptions, for each $M>0$ there exists $N\in \mathbb{N}$ such that for every $F\in \mathcal{F}^+$, we have
$$\frac{\Phi^{-1}(\frac{1}{\mu(F)})}{\Phi^{-1}(\frac{1}{\mu(\varphi^{-n}(F))})}>M, \ \ \ \ \ \ \text{for each} \ \ \  n\geq N.$$
By equation \ref{e1}, it is easy to see that for $f,g\in L^{\Phi}(\mu)$, with $S(f)\cap S(g)\neq \emptyset$, there exists $c>0$ such that
 $$\|f+g\|_{\Phi}\geq c(\|f\|_{\Phi}+\|g\|_{\Phi}).$$
 Now for every simple function $f\in S_{L^{\Phi}(\mu)}$, with $f=\Sigma^{m}_{i=1}\alpha_i\chi_{F_i}$, $\alpha_i\in \mathbb{C}\setminus \{0\}$ and $F_i\in \mathcal{F}$, without loss of generality we can assume that measurable sets $\{F_i\}^{m}_{i=1}$ are pairwise disjoint. Hence by the \ref{e2}, for any $n\geq N$, we have\\
 \begin{align*}
 N_{\Phi}(C^n_{\varphi}f)&=N_{\Phi}(\Sigma^{m}_{i=1}\alpha_i\chi_{F_i}\circ\varphi^n)\geq \frac{c}{2}\Sigma^{m}_{i=1}|\alpha_i| N_{\Phi}\chi_{\varphi^{-n}(F_i)}\\
 &=\frac{c}{2}\Sigma^{m}_{i=1}|\alpha_i|\frac{1}{\Phi^{-1}(\frac{1}{\mu(\varphi^{-n}(F_i))})}\geq \frac{cM}{2}\Sigma^{m}_{i=1}|\alpha_i| \frac{1}{\Phi^{-1}(\frac{1}{\mu(F_i)})}\\
 &=\frac{cM}{2}\Sigma^{m}_{i=1}|\alpha_i| N_{\Phi}(\chi_{F_i})\geq \frac{cM}{2}N_{\Phi}(f)=\frac{cM}{2}.
 \end{align*}

This implies that $\lim_{n\rightarrow \infty}N_{\Phi}(C^n_{\varphi}f)=\infty$. As is known, for every $f\in S_{L^{\Phi}(\mu)}$, we have a decomposition as
$f=f^+-f^-$,
in which $f^+, f^-$ are positive measurable functions. We can find a sequence of simple functions $\{f_k\}$ with $|f_k|\leq 2|f|$, such that it converges point-wisely to $f$ and so the sequence $\{\Phi(f_k)\}$ is also convergent point-wisely to $\Phi(f)$. Then by using Lebesgue Dominated Convergence Theorem we have
$$\lim_{k\rightarrow \infty} \rho_{\Phi}(f_k)=\rho_{\Phi}(f).$$
Hence by Theorem \ref{t1.1} we get that

$$\lim_{k\rightarrow \infty} N_{\Phi}(f_k)=N_{\Phi}(f)=1.$$
Therefore by combining the above observations we get that  $\lim_{n\rightarrow \infty}N_{\Phi}(C^n_{\varphi}f)=\infty$.\\
(4) Let $C_{\varphi}$ be uniformly expansive. Then by part (d) of the Proposition \ref{p1.1} we have a decomposition as $S_{L^{\Phi}(\mu)}=A\cup B$, where \\
$$\lim_{n\rightarrow \infty}N_{\Phi}(C^n_{\varphi}f)=\infty, \ \ \ \ \text{uniformly on} \ \ \ \  A$$
and
$$\lim_{n\rightarrow \infty}N_{\Phi}(C^{-n}_{\varphi}f)=\infty, \ \ \ \ \text{uniformly on} \ \ \ \  B.$$
If we put $\mathcal{F}^+=\{F\in\mathcal{F}: 0<\mu(F)<\infty\}$, then by the fact that simple functions are dense in $L^{\Phi}(\mu)$, we get: $\mathcal{F}^+=\mathcal{F}^+_A\cup \mathcal{F}^+_B$ in which
$$\mathcal{F}^+_A=\{F\in \mathcal{F}^+:\Phi^{-1}(\frac{1}{\mu(F)})\chi_{F}\in A\}\ \ \ \text{and} \ \ \ \ \mathcal{F}^+_B=\{F\in \mathcal{F}^+:\Phi^{-1}(\frac{1}{\mu(F)})\in B\}.$$
Since
$$\Phi^{-1}(\frac{1}{\mu(F)})\chi_{F}\in S_{L^{\Phi}(\mu)} \ \ \text{and} \ \ \ C^n_{\varphi}(\Phi^{-1}(\frac{1}{\mu(F)})\chi_{F})=\Phi^{-1}(\frac{1}{\mu(F)})\chi_{\varphi^{-n}(F)},$$
 then by the hypothesis we have
$$\lim_{n\rightarrow \infty}\frac{\Phi^{-1}(\frac{1}{\mu(F)}}{\Phi^{-1}(\frac{1}{\mu(\varphi^{n}(F))})}=\infty, \ \ \ \ \ \text{on} \ \ \ \mathcal{F}^+_A$$
and
$$\lim_{n\rightarrow \infty}\frac{\Phi^{-1}(\frac{1}{\mu(F)}}{\Phi^{-1}(\frac{1}{\mu(\varphi^{-n}(F))})}=\infty, \ \ \ \ \ \text{on} \ \ \ \mathcal{F}^+_B.$$
So we get the result. Now to prove the converse, we use again part (d) of the Proposition \ref{p1.1}. Hence it suffices to show that $S_{L^{\Phi}(\mu)}=A\cup B$, such that
 $$\lim_{n\rightarrow \infty}N_{\Phi}(C^n_{\varphi}f)=\infty, \ \ \ \ \text{uniformly on} \ \ \ \  A$$
and
$$\lim_{n\rightarrow \infty}N_{\Phi}(C^{-n}_{\varphi}f)=\infty, \ \ \ \ \text{uniformly on} \ \ \ \  B.$$
By the assumptions, for every $M>0$, there exists $N\in \mathbb{N}$ such that for every $n\geq N$,
$$\frac{\Phi^{-1}(\frac{1}{\mu(F)}}{\Phi^{-1}(\frac{1}{\mu(\varphi^{n}(F))})}>M\ \ \text{for each} \ F\in \mathcal{F}^+_A$$
and
$$ \frac{\Phi^{-1}(\frac{1}{\mu(F)}}{\Phi^{-1}(\frac{1}{\mu(\varphi^{-n}(F))})}>M\ \ \text{for each} \ F\in \mathcal{F}^+_B.$$
Now, let $f\in S_{L^{\Phi}(\mu)}$ be a simple function with $f=\sum_{i=1}^m\alpha_i\chi_{F_i}$, in which $F_i$'s are pairwise disjoint measurable sets with positive measure. Since $\mathcal{F}^+=\mathcal{F}^+_{\mathcal{A}}\cup \mathcal{F}^+_{\mathcal{B}}$, then we can define
$$f_{\mathcal{F}^+_{A}}=\sum_{\{1\leq i\leq m: F_i\in \mathcal{F}^+_{A}\}}\alpha_i\chi_{F_i}$$
and
$$f_{\mathcal{F}^+_{B}}=\sum_{\{1\leq i\leq m: F_i\in \mathcal{F}^+_{B}\}}\alpha_i\chi_{F_i}.$$
It is clear that $f=f_{\mathcal{F}^+_{A}}+f_{\mathcal{F}^+_{B}}$. So by equation \ref{e1} there exists $c>0$ such that
$$\|f\|_{\Phi}\geq c(\|f_{\mathcal{F}^+_{A}}\|_{\Phi}+\|f_{\mathcal{F}^+_{B}}\|_{\Phi})\geq c\|f\|_{\Phi}$$
and since $N_{\Phi}(f)=1$, then by the inequality \ref{e2} we have

$$2\geq c(\|f_{\mathcal{F}^+_{A}}\|_{\Phi}+\|f_{\mathcal{F}^+_{B}}\|_{\Phi})\geq 2c.$$
This means that
$$\|f_{\mathcal{F}^+_{A}}\|_{\Phi}\geq 1 \ \ \ \ \ \text{or} \ \ \ \ \ \|f_{\mathcal{F}^+_{B}}\|_{\Phi}\geq 1$$
and so by the inequality \ref{e2} we have
$$N_{\Phi}(f_{\mathcal{F}^+_{A}})\geq \frac{1}{2} \ \ \ \ \ \text{or} \ \ \ \ \ N_{\Phi}(f_{\mathcal{F}^+_{B}})\geq \frac{1}{2}.$$
Hence if $N_{\Phi}(f_{\mathcal{F}^+_{A}})\geq \frac{1}{2}$, then for every $n\geq N$
\begin{align*}
N_{\Phi}(C^{-n}_{\varphi}f)&\geq \frac{1}{2}\|C^{-n}_{\varphi}f\|_{\Phi}=\frac{1}{2}\|C^{-n}_{\varphi}(f_{\mathcal{F}^+_{A}}+f_{\mathcal{F}^+_{B}})\|_{\Phi}\\
&\geq \frac{c}{2}(\|C^{-n}_{\varphi}(f_{\mathcal{F}^+_{A}}\|_{\Phi}+\|f_{\mathcal{F}^+_{B}})\|_{\Phi})\geq \frac{c}{2} \|C^{-n}_{\varphi}f_{\mathcal{F}^+_{A}}\|_{\Phi}\\
&\geq \frac{c^2}{2}\sum_{\{1\leq i\leq m: F_i\in \mathcal{F}^+_{A}\}}|\alpha_i|\frac{1}{\Phi^{-1}(\frac{1}{\mu(\varphi^n(F_i))})}\\
&\geq \frac{c^2M}{2}\sum_{\{1\leq i\leq m: F_i\in \mathcal{F}^+_{A}\}}|\alpha_i|\frac{1}{\Phi^{-1}(\frac{1}{\mu(F_i)})}\\
&\geq \frac{c^2M}{2}N_{\Phi}(f_{\mathcal{F}^+_{A}})\geq \frac{c^2M}{4}.
\end{align*}
By a similar way we get that if $N_{\Phi}(f_{\mathcal{F}^+_{B}})\geq \frac{1}{2}$, then we have $N_{\Phi}(C^{n}_{\varphi}f)\geq \frac{c^2M}{4}$. If we set
$$S^0_{L^{\Phi}(\mu)}=\{f\in S_{L^{\Phi}(\mu)}: f \ \ \text{is a simple function}\},$$
then from the above observations we get that $S^0_{L^{\Phi}(\mu)}=A_0\cup B_0$, in which
$$A_0=\{f\in S^0_{L^{\Phi}(\mu)}: N_{\Phi}(f_{\mathcal{F}^+_{A}})\geq \frac{1}{2}\} \ \ \ \ \text{and} \ \ \ \ B_0=\{f\in S^0_{L^{\Phi}(\mu)}: N_{\Phi}(f_{\mathcal{F}^+_{B}})\geq \frac{1}{2}\}.$$
Let $f\in S_{L^{\Phi}(\mu)}$ be an arbitrary element. Then we can write $f=f^{+}-f^{-}$ in which $f^{+}$ and $f^{-}$ are positive and negative parts of $f$, respectively. Hence we can find two non-decreasing sequences of simple functions $\{f^+_k\}_{k\in \mathbb{N}}$ and $\{f^+-_k\}_{k\in \mathbb{N}}$ such that they converge pointwisely to $f^+$ and $f^-$, respectively. If we define $f_k=f^+_k-f^-_k$, then we get that the sequence $\{f_k\}_{k\in \mathbb{N}}$ converges pointwisely to $f$. Since $|f_k|=|f^+_k-f^-_k|\leq 2|f|$ and so $\Phi(|f_k|)\leq \Phi(2|f|)$, then by Lebesgue Dominated Convergence Theorem we have
$$\lim_{k\rightarrow \infty} \rho_{\Phi}(f_k)=\rho_{\Phi}(f).$$
Hence by Theorem \ref{t1.1} we get that

$$\lim_{k\rightarrow \infty} N_{\Phi}(f_k)=N_{\Phi}(f)=1.$$
So there exists $k_0$ such that for every $k\geq k_0$, $N_{\Phi}(f_k)> \frac{1}{2}$. By the above observations we get that
$$N_{\Phi}(C^n_{\varphi}(\frac{f_k}{N_{\Phi}(f_k)})<4N_{\Phi}(C^n_{\varphi}(f)$$
and also at least one of the following sets must be infinite:
$$I_1(f)=\{k\in \mathbb{N}: \frac{f_k}{N_{\Phi}(f_k)}\in A_0\}; \ \ \ \ I_1(f)=\{k\in \mathbb{N}: \frac{f_k}{N_{\Phi}(f_k)}\in B_0\}.$$
Therefore, if $I_1(f)$ is infinite, then there is a subsequence
$$\{ \frac{f_{k_j}}{N_{\Phi}(f_{k_j})}\}_{j\in \mathbb{N}}\subseteq A_0$$
such that for every $M>0$, there exists $N\in \mathbb{N}$ such that for each $n\geq N$, $$N_{\Phi}(C^{-n}_{\varphi}(\frac{f_{k_j}}{N_{\Phi}(f_{k_j})})>\frac{M}{2}.$$
Thus we have
$$N_{\Phi}(C^{-n}_{\varphi}(f))>\frac{M}{8}, \ \ \ \ \ \text{for every} \ \ \ n\geq N.$$
Similarly, if $I_2(f)=\infty$, then for each $M>0$ we can find $N\in \mathbb{N}$, such that for every $n\geq N$, $N_{\Phi}(C^n_{\varphi}(f))> \frac{M}{8}$.
Finally if we set
$$A=\{f\in S_{L^{\Phi}(\mu)}:I_2(f)=\infty\} \ \ \ \text{and} \ \ \ B=\{f\in S_{L^{\Phi}(\mu)}:I_1(f)=\infty\},$$
then we have $S_{L^{\Phi}(\mu)}=A\cup B$, such that
$$\lim_{n\rightarrow \infty}N_{\Phi}(C^n_{\varphi}f)=\infty, \ \ \ \ \text{uniformly on} \ \ \ \  A$$
and
$$\lim_{n\rightarrow \infty}N_{\Phi}(C^{-n}_{\varphi}f)=\infty, \ \ \ \ \text{uniformly on} \ \ \ \  B.$$
This completes the proof.
\end{proof}

Now by using similar methods of Theorem \ref{t2.2} we have the next Theorem.

\begin{thm}\label{tt2.8} Let $(X, \mathcal{F}, \mu,\varphi, C_{\varphi})$ be a dissipative composition dynamical system of bounded distortion, generated by $W$. Then the following statements hold.
\begin{enumerate}
  \item $C_{\varphi}$ is positively expansive if and only if $\inf_{n\in\mathbb{N}}\Phi^{-1}(\frac{1}{\mu(\varphi^{-n}(W))})=0$.
  \item $C_{\varphi}$ is expansive if and only if $\inf_{n\in\mathbb{Z}}\Phi^{-1}(\frac{1}{\mu(\varphi^{n}(W))})=0$.
  \item $C_{\varphi}$ is uniformly positively expansive if and only if $$\lim_{n\rightarrow \infty}\Phi^{-1}(\frac{1}{\mu(\varphi^{-n}(W))})=0.$$
  \item $C_{\varphi}$ is uniformly expansive if and only if one of the following conditions hold:
  \begin{equation}\label{ue1}
    \lim_{n\rightarrow\infty}\inf_{k\in\mathbb{Z}}\left(\frac{\Phi^{-1}(\frac{1}{\mu(\varphi^k(W))})}{\Phi^{-1}(\frac{1}{\mu(\varphi^{k+n}(W))})}\right)=\infty,
  \end{equation}
  \begin{equation}\label{ue2}
    \lim_{n\rightarrow\infty}\inf_{k\in\mathbb{Z}}\left(\frac{\Phi^{-1}(\frac{1}{\mu(\varphi^k(W))})}{\Phi^{-1}(\frac{1}{\mu(\varphi^{k-n}(W))})}\right)=\infty,
  \end{equation}
  \begin{equation}\label{ue3}
    \lim_{n\rightarrow\infty}\inf_{k\in\mathbb{Z}}\left(\frac{\Phi^{-1}(\frac{1}{\mu(\varphi^k(W))})}{\Phi^{-1}(\frac{1}{\mu(\varphi^{k+n}(W))})}\right)=\infty \ \ \ \text{and}\ \ \
    \lim_{n\rightarrow\infty}\inf_{k\in-\mathbb{N}_0}\left(\frac{\Phi^{-1}(\frac{1}{\mu(\varphi^k(W))})}{\Phi^{-1}(\frac{1}{\mu(\varphi^{k-n}(W))})}\right)=\infty.
  \end{equation}
\end{enumerate}
\end{thm}
\begin{proof}
First we prove part (2), then for the proof of part (1) it suffices just to replace $\mathbb{Z}$ by $\mathbb{N}$.\\
In Theorem \ref{t2.2} we can replace $\mathcal{F}$ by $\mathcal{F}^+$, i.e., we can write "the composition operator $C_{\varphi}$ is expansive if and only if for every $A\in \mathcal{F}^+$,
  $$\inf_{n\in \mathcal{Z}}\Phi^{-1}(\frac{1}{\mu(\varphi^{-n}(A))})=0.$$
  Hence for the proof of (2), it suffices to prove
  $$\inf_{n\in \mathcal{Z}}\Phi^{-1}(\frac{1}{\mu(\varphi^{-n}(A))})=0, \ \ \ \forall A\in \mathcal{F}^+\ \ \Leftrightarrow \ \ \inf_{n\in\mathbb{Z}}\Phi^{-1}(\frac{1}{\mu(\varphi^{n}(W))})=0.$$
  The implication
   $$\inf_{n\in \mathcal{Z}}\Phi^{-1}(\frac{1}{\mu(\varphi^{-n}(A))})=0, \ \ \ \forall A\in \mathcal{F}^+\ \ \Rightarrow \ \ \inf_{n\in\mathbb{Z}}\Phi^{-1}(\frac{1}{\mu(\varphi^{n}(W))})=0$$
   is clear. So we prove the converse.
Let $F\in\mathcal{F}^+$. Since $X=\cup_{n\in\mathbb{Z}}\varphi^n(W)$, in which $\varphi^n(W)$'s are pair-wise disjoint and $\mu(F)>0$, then there exists $n_0\in \mathbb{Z}$ such that $\mu(F\cap\varphi^{n_0}(W))>0$. On the other hand $\varphi^{-n_0}(F\cap \varphi^{n_0}(W))=\varphi^{n_0}(F)\cap W$. Hence if we take $A=\varphi^{-n_0}(F\cap \varphi^{n_0}(W)$, then $A\in \mathcal{F}_W$ and $\mu(A)\leq \mu(W)$. Now by using bounded distortion property for
 we have
$$\Phi^{-1}(\frac{1}{\mu(\varphi^{-n}(A))})\leq \frac{KN_{\Phi}(\chi_W)}{N_{\Phi}(\chi_A)}\Phi^{-1}(\frac{1}{\mu(\varphi^{-n}(W))}), \ \ \forall n\in \mathbb{Z}$$
and therefore
$$\inf_{n\in\mathbb{Z}}\Phi^{-1}(\frac{1}{\mu(\varphi^{-n}(A))})\leq \frac{KN_{\Phi}(\chi_W)}{N_{\Phi}(\chi_A)}\inf_{n\in\mathbb{Z}}\Phi^{-1}(\frac{1}{\mu(\varphi^{-n}(W))})=0.$$
Consequently we have
$$\inf_{n\in\mathbb{Z}}\Phi^{-1}(\frac{1}{\mu(\varphi^{-n}(F))})\leq \inf_{n\in\mathbb{Z}}\Phi^{-1}(\frac{1}{\mu(\varphi^{-n}(A))})=0.$$
(3). By Theorem \ref{t2.2} it suffices to prove that
$$\lim_{n\rightarrow \infty}\frac{\Phi^{-1}(\frac{1}{\mu(F)})}{\Phi^{-1}(\frac{1}{\mu(\varphi^{-n}(F))})}=\infty, \ \ \forall F\in \mathcal{F}^+ \ \ \Leftrightarrow \ \ \lim_{n\rightarrow \infty}\Phi^{-1}(\frac{1}{\mu(\varphi^{-n}(W))})=0.$$
It is obvious that if $$\lim_{n\rightarrow \infty}\frac{\Phi^{-1}(\frac{1}{\mu(F)})}{\Phi^{-1}(\frac{1}{\mu(\varphi^{-n}(F))})}=\infty, \ \ \forall F\in \mathcal{F}^+,$$
then $\lim_{n\rightarrow \infty}\frac{1}{\Phi^{-1}(\frac{1}{\mu(\varphi^{-n}(W))})}=\infty$ and so $\lim_{n\rightarrow \infty}\Phi^{-1}(\frac{1}{\mu(\varphi^{-n}(W))})=0$. Now we prove the converse, suppose that $\lim_{n\rightarrow \infty}\Phi^{-1}(\frac{1}{\mu(\varphi^{-n}(W))})=0$ and $F\in \mathcal{F}^+$. Similar to pat (2) we can find $n_0\in \mathbb{Z}$ such that $0<\mu(F\cap \varphi^{n_0}(W))<\infty$ and
$$A=\varphi^{-n_0}(F\cap \varphi^{n_0}(W))=\varphi^{-n_0}(F)\cap W.$$

 Hence $0<\mu(A)<\infty$ and $A\in \mathcal{F}_W$. Now by using bounded distortion property we have
$$\Phi^{-1}(\frac{1}{\mu(\varphi^{-n}(A))})\leq \frac{KN_{\Phi}(\chi_W)}{N_{\Phi}(\chi_A)}\Phi^{-1}(\frac{1}{\mu(\varphi^{-n}(W))}), \ \ \ \forall n\in \mathbb{Z}$$
and so
$$\Phi^{-1}(\frac{1}{\mu(\varphi^{-n-n_0}(F))})\leq\Phi^{-1}(\frac{1}{\mu(\varphi^{-n}(A))})\leq \frac{KN_{\Phi}(\chi_W)}{N_{\Phi}(\chi_A)}\Phi^{-1}(\frac{1}{\mu(\varphi^{-n}(W))})$$
Thus by taking the limit of each sides of the above inequality, we get
 $$\lim_{n\rightarrow \infty}\Phi^{-1}(\frac{1}{\mu(\varphi^{-n-n_0}(F))})\leq \frac{KN_{\Phi}(\chi_W)}{N_{\Phi}(\chi_A)}\lim_{n\rightarrow \infty}\Phi^{-1}(\frac{1}{\mu(\varphi^{-n}(W))})=0$$
 hence $\lim_{n\rightarrow \infty}\Phi^{-1}(\frac{1}{\mu(\varphi^{-n-n_0}(F))})=0$
  and consequently
  $$\lim_{n\rightarrow \infty}\frac{\Phi^{-1}(\frac{1}{\mu(F)})}{\Phi^{-1}(\frac{1}{\mu(\varphi^{-n}(F))})}=\infty.$$
  By Theorem \ref{t2.2}, part (3) we get the proof.\\
  (4) Suppose that $C_{\varphi}$ is uniformly expansive. By Theorem \ref{t2.2} we have that
  $\mathcal{F}^+=\mathcal{F}^+_A\cup \mathcal{F}^+_B$ such that
      $$\lim_{n\rightarrow \infty}\frac{\Phi^{-1}(\frac{1}{\mu(F)})}{\Phi^{-1}(\frac{1}{\mu(\varphi^{n}(F))})}=\infty, \ \ \ \text{uniformly on } \ \mathcal{F}^+_A,$$
       $$\lim_{n\rightarrow \infty}\frac{\Phi^{-1}(\frac{1}{\mu(F)})}{\Phi^{-1}(\frac{1}{\mu(\varphi^{-n}(F))})}=\infty, \ \ \ \text{uniformly on } \ \mathcal{F}^+_B.$$
       If we set $$I=\{k\in \mathbb{Z}: \varphi^k(W)\in \mathcal{F}_A\}\ \ \ \text{and} \ \ \ J=\{k\in \mathbb{Z}: \varphi^k(W)\in \mathcal{F}_B\},$$
       then by our assumptions we get that
       $$\lim_{n\rightarrow\infty}\inf_{k\in I}\left(\frac{\Phi^{-1}(\frac{1}{\mu(\varphi^k(W))})}{\Phi^{-1}(\frac{1}{\mu(\varphi^{k+n}(W))})}\right)=\infty$$

  $$\lim_{n\rightarrow\infty}\inf_{k\in J}\left(\frac{\Phi^{-1}(\frac{1}{\mu(\varphi^k(W))})}{\Phi^{-1}(\frac{1}{\mu(\varphi^{k-n}(W))})}\right)=\infty.$$
  By these observations, if $J=\emptyset$, then the condition \ref{ue1} holds, if $I=\emptyset$, then the condition \ref{ue2} holds and if $I$ and $J$ are both non-empty, then there exist $i,j\in \mathbb{Z}$ such that $[i,\infty)\cap\mathbb{Z}\subseteq I$ and $(-\infty,j]\cap\mathbb{Z}\subseteq J$, this implies that the condition \ref{ue3} holds.\\
  For the converse, for every $F\in \mathcal{F}^+$, by our assumptions we have $F=\cup_{k\in \mathbb{Z}}F_k$, where $F_k=F\cap \varphi^k(W)$ and $F_k$'s are pair-wise disjoint and also by Proposition \ref{p2.6}, for every $n, k\in \mathbb{Z}$ there exists $H>0$ such that
  $$\frac{1}{H}\frac{N_{\Phi}(C^{k+n}_{\varphi}(\chi_W))}{N_{\Phi}(C^k_\varphi(\chi_W))}N_{\Phi}(\chi_{F_k})\leq N_{\Phi}(C^n_{\varphi}(\chi_{F_k}))\leq H\frac{N_{\Phi}(C^{k+n}_{\varphi}(\chi_W))}{N_{\Phi}(C^k_\varphi(\chi_W))}N_{\Phi}(\chi_{F_k}).$$
  Also, we have $0<\mu(F)<\infty$. Hence there exists $k\in \mathbb{Z}$ such that $\mu(F_k)>0$. Therefore , for each $n\in \mathbb{Z}$,

  \begin{align*}
  N_{\Phi}(C^n_{\varphi}(\chi_F))&\geq \frac{1}{2}\|C^n_{\varphi}(\chi_F)\|_{\Phi}=\frac{1}{2}\|\chi_{\varphi^{-n}(F)}\|_{\Phi}\\
  &\geq\frac{c}{2}\sum_{k\in \mathbb{Z}}\|\chi_{\varphi^{-n}(F_k)}\|_{\Phi}\geq\frac{c}{2}\sum_{k\in \mathbb{Z}}N_{\Phi}(\chi_{\varphi^{-n}(F_k)})\\
  &=\frac{c}{2}\sum_{k\in \mathbb{Z}}N_{\Phi}(C^n_{\varphi}(\chi_{F_k}))\geq\frac{c}{2H}\sum_{k\in \mathbb{Z}}\frac{N_{\Phi}(C^{k+n}_{\varphi}(\chi_W))}{N_{\Phi}(C^k_\varphi(\chi_W))}N_{\Phi}(\chi_{F_k})\\
  &\geq\frac{c}{2H}\inf_{k\in\mathbb{Z}}\frac{N_{\Phi}(C^{k+n}_{\varphi}(\chi_W))}{N_{\Phi}(C^k_\varphi(\chi_W))}\sum_{k\in \mathbb{Z}}N_{\Phi}(\chi_{F_k})\geq\frac{c}{2H}\inf_{k\in\mathbb{Z}}\frac{N_{\Phi}(C^{k+n}_{\varphi}(\chi_W))}{N_{\Phi}(C^k_\varphi(\chi_W))}N_{\Phi}(\chi_{F}).
    \end{align*}
By these observations we get that
$$\frac{\Phi^{-1}(\frac{1}{\mu(F)})}{\Phi^{-1}(\frac{1}{\mu(\varphi^{n}(F))})}\geq\frac{c}{2H}\inf_{k\in\mathbb{Z}}\frac{\Phi^{-1}(\frac{1}{\mu(\varphi^{k}(W))})}{\Phi^{-1}(\frac{1}{\mu(\varphi^{k+n}(W))})}.$$

This inequality implies that if we set
$$\mathcal{F}^+_A=\{F\cap(\cup_{k\geq0}\varphi^k(W)): F\in \mathcal{F}^+\}$$
and
$$\mathcal{F}^+_B=\{F\cap(\cup_{k<0}\varphi^k(W)): F\in \mathcal{F}^+\},$$
then $\mathcal{F}^+=\mathcal{F}^+_A\cup\mathcal{F}^+_B$ and also if one of conditions \ref{ue1}, \ref{ue2} or \ref{ue3} holds, then we get that

$$\lim_{n\rightarrow \infty}\frac{\Phi^{-1}(\frac{1}{\mu(F)})}{\Phi^{-1}(\frac{1}{\mu(\varphi^{n}(F))})}=\infty, \ \ \ \text{uniformly on } \ \mathcal{F}^+_A,$$

$$\lim_{n\rightarrow \infty}\frac{\Phi^{-1}(\frac{1}{\mu(F)})}{\Phi^{-1}(\frac{1}{\mu(\varphi^{-n}(F))})}=\infty, \ \ \ \text{uniformly on } \ \mathcal{F}^+_B.$$
Thus by Theorem \ref{t2.2} the composition operator $C_{\varphi}$ is uniformly expansive.
\end{proof}

In the next Proposition we obtain some sufficient conditions for composition operator $C_{\varphi}$ on Orlicz spaces to decline structural stability and even strongly structurally stability.
\begin{prop}\label{p2.9}
Let $(X, \mathcal{F},\mu, \varphi, C_{\varphi})$ be a dissipative composition dynamical system of bounded distortion, generated by $W$ and $\Phi\in \Delta'$. If
$$\overline{\lim}_{n\rightarrow\infty}\sup_{k\in\mathbb{N}_0}\left(\frac{\Phi^{-1}(\frac{1}{\mu(\varphi^{k+n}(W))})}{\Phi^{-1}(\frac{1}{\mu(\varphi^k(W))})}\right)^{\frac{1}{n}}<1 \ \ \ \text{and} \ \ \ \underline{\lim}_{n\rightarrow \infty}\inf_{k\in-\mathbb{N}_0}\left(\frac{\Phi^{-1}(\frac{1}{\mu(\varphi^{k}(W))})}{\Phi^{-1}(\frac{1}{\mu(\varphi^{k-n}(W))})}\right)^{\frac{1}{n}}>1,$$
then the composition operator $C_{\varphi}$ can not be structurally stable and so can not be even strongly structurally stable.
\end{prop}
\begin{proof}
Since $$\underline{\lim}_{n\rightarrow \infty}\inf_{k\in-\mathbb{N}_0}\left(\frac{\Phi^{-1}(\frac{1}{\mu(\varphi^{k}(W))})}{\Phi^{-1}(\frac{1}{\mu(\varphi^{k-n}(W))})}\right)^{\frac{1}{n}}>1,$$
then there exists $n_0\in \mathbb{N}_0$ and $t>1$ such that for every $n\geq n_0$,
$$\inf_{k\in-\mathbb{N}_0}\left(\frac{\Phi^{-1}(\frac{1}{\mu(\varphi^{k}(W))})}{\Phi^{-1}(\frac{1}{\mu(\varphi^{k-n}(W))})}\right)^{\frac{1}{n}}\geq t>1.$$
Hence
$$\frac{\Phi^{-1}(\frac{1}{\mu(W)})}{\Phi^{-1}(\frac{1}{\mu(\varphi^{-n}(W))})}\geq \inf_{k\in-\mathbb{N}_0}\frac{\Phi^{-1}(\frac{1}{\mu(\varphi^{k}(W))})}{\Phi^{-1}(\frac{1}{\mu(\varphi^{k-n}(W))})}\geq t^n.$$
This implies that $\sup_{n\in \mathbb{N}}\frac{\Phi^{-1}(\frac{1}{\mu(W)})}{\Phi^{-1}(\frac{1}{\mu(\varphi^{-n}(W))})}=\infty$ and therefore
$$\inf_{n\in \mathbb{N}}\Phi^{-1}(\frac{1}{\mu(\varphi^{-n}(W))})=0.$$
Thus by theorem \ref{t2.2}, we get that $C_{\varphi}$ is positively expansive. Moreover, for every $n\geq n_0$,
$$\sup_{k\in-\mathbb{N}_0}\left(\frac{\Phi^{-1}(\frac{1}{\mu(\varphi^{k}(W))})}{\Phi^{-1}(\frac{1}{\mu(\varphi^{k-n}(W))})}\right)^{\frac{1}{n}}\geq\inf_{k\in-\mathbb{N}_0}\left(\frac{\Phi^{-1}(\frac{1}{\mu(\varphi^{k}(W))})}{\Phi^{-1}(\frac{1}{\mu(\varphi^{k-n}(W))})}\right)^{\frac{1}{n}}\geq t>1.$$
Hence we get that
$$\overline{\lim}_{n\rightarrow \infty}\sup_{k\in \mathbb{Z}}\left(\frac{\Phi^{-1}(\frac{1}{\mu(\varphi^{k}(W))})}{\Phi^{-1}(\frac{1}{\mu(\varphi^{k-n}(W))})}\right)^{\frac{1}{n}}\geq \overline{\lim}_{n\rightarrow \infty}\sup_{k\in -\mathbb{N}_0}\left(\frac{\Phi^{-1}(\frac{1}{\mu(\varphi^{k}(W))})}{\Phi^{-1}(\frac{1}{\mu(\varphi^{k-n}(W))})}\right)^{\frac{1}{n}}>1.$$
In addition, since $$\overline{\lim}_{n\rightarrow\infty}\sup_{k\in\mathbb{N}_0}\left(\frac{\Phi^{-1}(\frac{1}{\mu(\varphi^{k+n}(W))})}{\Phi^{-1}(\frac{1}{\mu(\varphi^k(W))})}\right)^{\frac{1}{n}}<1,$$
then there exists $m_0\in \mathbb{N}_0$ and $0<t<1$ such that for every $n\geq m_0$,
$$\sup_{k\in\mathbb{N}_0}\left(\frac{\Phi^{-1}(\frac{1}{\mu(\varphi^{k+n}(W))})}{\Phi^{-1}(\frac{1}{\mu(\varphi^k(W))})}\right)^{\frac{1}{n}}\leq t<1.$$
This implies that
$$\inf_{k\in\mathbb{Z}}\left(\frac{\Phi^{-1}(\frac{1}{\mu(\varphi^{k+n}(W))})}{\Phi^{-1}(\frac{1}{\mu(\varphi^k(W))})}\right)^{\frac{1}{n}}\leq\inf_{k\in\mathbb{N}_0}\left(\frac{\Phi^{-1}(\frac{1}{\mu(\varphi^{k+n}(W))})}{\Phi^{-1}(\frac{1}{\mu(\varphi^k(W))})}\right)^{\frac{1}{n}}\leq t<1$$
and consequently
$$\underline{\lim}_{n\rightarrow \infty}\inf_{k\in\mathbb{Z}}\left(\frac{\Phi^{-1}(\frac{1}{\mu(\varphi^{k+n}(W))})}{\Phi^{-1}(\frac{1}{\mu(\varphi^k(W))})}\right)^{\frac{1}{n}}\leq \underline{\lim}_{n\rightarrow \infty}\inf_{k\in\mathbb{N}_0}\left(\frac{\Phi^{-1}(\frac{1}{\mu(\varphi^{k+n}(W))})}{\Phi^{-1}(\frac{1}{\mu(\varphi^k(W))})}\right)^{\frac{1}{n}}<1.$$

Based on the above observations, we can conclude that none of conditions 2.3, 2.4, and 2.5 of Corollary 2.20 in \cite{eb} are met. This implies that $C_{\varphi}$ does not possess the shadowing property and, therefore, is not hyperbolic. Consequently, $C_{\varphi}$ is positively expansive but not hyperbolic. According to Theorem 6 in \cite{bme}, this means that $C_{\varphi}$ is not structurally stable and, consequently, not even strongly structurally stable.
\end{proof}
\begin{thm}\label{tt2.10}
Let $(X, \mathcal{F},\mu, \varphi, C_{\varphi})$ be a dissipative composition dynamical system of bounded distortion, generated by $W$ and $\Phi\in \Delta'$. If at least one of the following inequalities holds:
\begin{equation}\label{hc}
 \overline{\lim}_{n\rightarrow\infty}\sup_{k\in\mathbb{Z}}\left(\frac{\Phi^{-1}(\frac{1}{\mu(\varphi^{k+n}(W))})}{\Phi^{-1}(\frac{1}{\mu(\varphi^{k}(W))})}\right)^{\frac{1}{n}}<1 \end{equation}
  \begin{equation}\label{hd} \underline{\lim}_{n\rightarrow\infty}\inf_{k\in\mathbb{Z}}\left(\frac{\Phi^{-1}(\frac{1}{\mu(\varphi^{k+n}(W))})}{\Phi^{-1}(\frac{1}{\mu(\varphi^{k}(W))})}\right)^{\frac{1}{n}}>1 \end{equation}
\begin{equation}\label{gh}
\overline{\lim}_{n\rightarrow\infty}\sup_{k\in-\mathbb{N}_0}\left(\frac{\Phi^{-1}(\frac{1}{\mu(\varphi^{k}(W))})}{\Phi^{-1}(\frac{1}{\mu(\varphi^{k-n}(W))})}\right)^{\frac{1}{n}}<1
 \ \ \text{and} \ \ \underline{\lim}_{n\rightarrow\infty}\inf_{k\in\mathbb{N}_0}\left(\frac{\Phi^{-1}(\frac{1}{\mu(\varphi^{k}(W))})}{\Phi^{-1}(\frac{1}{\mu(\varphi^{k+n}(W))})}\right)^{\frac{1}{n}}>1, \end{equation}
 then the composition operator $C_{\varphi}$ is strongly structurally stable.
\end{thm}
\begin{proof}
If one of the conditions \ref{hc}, \ref{hd} and \ref{gh} holds, then by Theorem 2.16 of \cite{eb}
we get that $C_{\varphi}$ is strongly structurally stable.
\end{proof}
Now by using Theorem \ref{tt2.10} and also Corollary CSC of \cite{eb} we have the next corollary.
\begin{cor}\label{csc1}
Let $(X, \mathcal{F},\mu, \varphi, C_{\varphi})$ be a dissipative composition dynamical system of bounded distortion, generated by $W$ and $\Phi\in \Delta'$. Then the composition operator $C_{\varphi}$ is strongly structurally stable when it has shadowing property or equivalently is generalized hyperbolic.
\end{cor}

\begin{thm}
Let  $(X, \mathcal{F},\mu, \varphi, C_{\varphi})$ be a dissipative composition dynamical system of bounded distortion, generated by $W$, $\inf_{n\in\mathbb{N}}\Phi^{-1}(\frac{1}{\mu(\varphi^{-n}(W))})=0$ and $\Phi\in \Delta'$. Then the composition operator $C_{\varphi}$ is strongly structurally stable if and only if it has the shadowing property.
\end{thm}
\begin{proof}
If $C_{\varphi}$ has the shadowing property, then by Corollary \ref{csc1} we get that $C_{\varphi}$ is strongly structurally stable. Conversely, if $C_{\varphi}$ is strongly structurally stable, then by the assumption $\inf_{n\in\mathbb{N}}\Phi^{-1}(\frac{1}{\mu(\varphi^{-n}(W))})=0$ and by Theorem \ref{tt2.8} we have $C_{\varphi}$ is positively expansive. Therefore by Theorem 6 of \cite{bme} we get the proof.
\end{proof}
\textbf{Declarations}\\
     \textbf{Acknowledgement.} My manuscript has no associate data.


\begin{thebibliography}{99}
\bibitem{aa} I, Akbarbaglu. M. R. Azimi, Universal family of translations on weighted Orlicz spaces, Positivity. 26 (2022). DOI:10.1007/s11117-022-00869-2.
\bibitem{af} M. R. Azimi, M. Farmani, Subspace-supercyclicity of conditional weighted type translations on $L^p(G)$. Adv. Oper. Theory, 37 (2023). https://doi.org/10.1007/s43036-023-00266-w.
    \bibitem{aj} M. R. Azimi, M.R, Jabbarzadeh, Hypercyclicity of Weighted Composition Operators on $L^p$-Spaces. Mediterr. J. Math. 164 (2022). https://doi.org/10.1007/s00009-022-02086-3.
\bibitem{bm} F. Bayart, É. Matheron, Dynamics of Linear Operators, Cambridge University Press, Cambridge, 2009.
\bibitem{bme} N. Bernardes and A. Messaoudi, Shadowing and structural stability for operators. Ergodic Theory and Dynamical Systems, 41 (2021), 961-980.
\bibitem{bmpp} J. Bès, Q. Menet, A. Peris, Y. Puig, Recurrence properties of hypercyclic operators, Math. Ann. 366 (1) (2016) 545–572.

\bibitem{br} F. Bayart, I.Z. Ruzsa, Difference sets and frequently hypercyclic weighted shifts, Ergodic Theory Dynam. Systems 35 (3)
(2015) 691–709.
\bibitem{chkm} Y. Cui, H. Hudzik, R. Kumar and L. Maligranda, Composition operators in Olicz spaces, J. Aust. Math. Soc.76 (2004), 189-206.
\bibitem{ddm} E . D’Aniello, U. B. Darji and M. Maiuriello,  Generalized hyperbolicity and shadowing in $L^p$ spaces. J.Difer. Equ. 298 (2021), 68–94.
\bibitem{eh}  M. Eisenberg, J.H. Hedlund, Expansive automorphisms of Banach spaces, Pacific J. Math. 34 (1970) 647–656.
\bibitem{eb} Y. Estaremi, Hyperbolic composition operators on Orlicz spaces, arxive.org.
\bibitem{ge} K.-G. Grosse-Erdmann, A. Peris Manguillot, Linear Chaos, Springer-Verlag, London, 2011.
\bibitem{gm} S. Grivaux, É. Matheron, Invariant measures for frequently hypercyclic operators, Adv. Math. 265 (2014) 371–427.
\bibitem{j} M. R. Jabbarzadeh, The essential norm of a composition operator on Orlicz spaces, Turkish Journal of Mathematics, 34 (2010). https://doi.org/10.3906/mat-0904-21.
\bibitem{jh} M. R. Jabbarzadeh, S. Haghighatjoo, Equivalent metrics on normal composition operators, Rocky Mountain J. Math. 50 (2020), 989-999.
\bibitem{kr}  M. A. Krasnosel'skii,  Ya. B. Rutickii, Convex functions and Orlicz spaces, Noordhoff, Netherlands, 1961.
\bibitem{mm} M. Maiuriello, Expansivity and strong structural stability for composition operators on Lp  spaces, Banach J. Math. Anal. (2022). https://doi.org/10.1007/s43037-022-00196-4



\bibitem{nbpc} C. Nilson Jr. Bernardes, R, C. Patricia, B. Udayan B. D. A. Messaoudi, R. P. Enrique, Expansivity and shadowing in linear dynamics, J. Math. Anal. Appl. 461 (2018) 796–816.
\bibitem{raor} M. M. Rao, Z.D. Ren, Theory of Orlicz spaces, Marcel Dekker,
New York, 1991.
\bibitem{sm} R.K. Singh, J.S. Manhas, Composition Operators on Function Spaces, North-Holland Mathematics Studies, vol. 179, North-Holland Publishing Co., Amsterdam, 1993.

\end{thebibliography}
\end{document}